\documentclass{amsart}
\usepackage{amssymb,amsmath,stmaryrd,mathrsfs}

\def\definetac{\newif\iftac}    % Can't define a \newif inside another \if!
\ifx\tactrue\undefined
  \definetac
  \ifx\state\undefined\tacfalse\else\tactrue\fi
\fi

\def\definebeamer{\newif\ifbeamer}
\ifx\beamertrue\undefined
  \definebeamer
  \ifx\uncover\undefined\beamerfalse\else\beamertrue\fi
\fi

\def\definecref{\newif\ifcref}
\ifx\creftrue\undefined
  \definecref
  \creffalse
\fi

\iftac\else\usepackage{amsthm}\fi
\usepackage{tikz,tikz-cd}
\usetikzlibrary{arrows}
\ifbeamer\else
  \usepackage{enumitem}
  \usepackage{xcolor}
  \definecolor{darkgreen}{rgb}{0,0.45,0} 
  \iftac\else\usepackage[pagebackref,colorlinks,citecolor=darkgreen,linkcolor=darkgreen]{hyperref}\fi
  \renewcommand*{\backref}[1]{}
  \renewcommand*{\backrefalt}[4]{({%
      \ifcase #1 Not cited.%
            \or On p.~#2%
            \else On pp.~#2%
      \fi%
    })}
\fi
\usepackage{mathtools}          % for all sorts of things
\usepackage{graphics}           % for \scalebox, used in \widecheck
\usepackage{ifmtarg}            % used in \jd
\usepackage{microtype}
\usepackage{braket}             % for \Set, etc.

\usepackage{url}                % for citations to web sites
\usepackage{xspace}             % put spaces after a \command in text
\ifcref\usepackage{cleveref,aliascnt}\fi
\usepackage[status=draft,author=]{fixme}
\fxusetheme{color}
\usepackage{mathpartir}

\makeatletter
\let\ea\expandafter

\def\mdef#1#2{\ea\ea\ea\gdef\ea\ea\noexpand#1\ea{\ea\ensuremath\ea{#2}\xspace}}
\def\alwaysmath#1{\ea\ea\ea\global\ea\ea\ea\let\ea\ea\csname your@#1\endcsname\csname #1\endcsname
  \ea\def\csname #1\endcsname{\ensuremath{\csname your@#1\endcsname}\xspace}}

\DeclareRobustCommand\widecheck[1]{{\mathpalette\@widecheck{#1}}}
\def\@widecheck#1#2{%
    \setbox\z@\hbox{\m@th$#1#2$}%
    \setbox\tw@\hbox{\m@th$#1%
       \widehat{%
          \vrule\@width\z@\@height\ht\z@
          \vrule\@height\z@\@width\wd\z@}$}%
    \dp\tw@-\ht\z@
    \@tempdima\ht\z@ \advance\@tempdima2\ht\tw@ \divide\@tempdima\thr@@
    \setbox\tw@\hbox{%
       \raise\@tempdima\hbox{\scalebox{1}[-1]{\lower\@tempdima\box
\tw@}}}%
    {\ooalign{\box\tw@ \cr \box\z@}}}

\newcount\foreachcount

\def\foreachletter#1#2#3{\foreachcount=#1
  \ea\loop\ea\ea\ea#3\@alph\foreachcount
  \advance\foreachcount by 1
  \ifnum\foreachcount<#2\repeat}

\def\foreachLetter#1#2#3{\foreachcount=#1
  \ea\loop\ea\ea\ea#3\@Alph\foreachcount
  \advance\foreachcount by 1
  \ifnum\foreachcount<#2\repeat}

\def\definescr#1{\ea\gdef\csname s#1\endcsname{\ensuremath{\mathscr{#1}}\xspace}}
\foreachLetter{1}{27}{\definescr}
\def\definecal#1{\ea\gdef\csname c#1\endcsname{\ensuremath{\mathcal{#1}}\xspace}}
\foreachLetter{1}{27}{\definecal}
\def\definebold#1{\ea\gdef\csname b#1\endcsname{\ensuremath{\mathbf{#1}}\xspace}}
\foreachLetter{1}{27}{\definebold}
\def\definebb#1{\ea\gdef\csname d#1\endcsname{\ensuremath{\mathbb{#1}}\xspace}}
\foreachLetter{1}{27}{\definebb}
\def\definefrak#1{\ea\gdef\csname f#1\endcsname{\ensuremath{\mathfrak{#1}}\xspace}}
\foreachletter{1}{9}{\definefrak}
\foreachletter{10}{27}{\definefrak}
\foreachLetter{1}{27}{\definefrak}
\def\definesf#1{\ea\gdef\csname i#1\endcsname{\ensuremath{\mathsf{#1}}\xspace}}
\foreachletter{1}{6}{\definesf}
\foreachletter{7}{14}{\definesf}
\foreachletter{15}{27}{\definesf}
\foreachLetter{1}{27}{\definesf}
\def\definebar#1{\ea\gdef\csname #1bar\endcsname{\ensuremath{\overline{#1}}\xspace}}
\foreachLetter{1}{27}{\definebar}
\foreachletter{1}{8}{\definebar} % \hbar is something else!
\foreachletter{9}{15}{\definebar} % \obar is something else!
\foreachletter{16}{27}{\definebar}
\def\definetil#1{\ea\gdef\csname #1til\endcsname{\ensuremath{\widetilde{#1}}\xspace}}
\foreachLetter{1}{27}{\definetil}
\foreachletter{1}{27}{\definetil}
\def\definehat#1{\ea\gdef\csname #1hat\endcsname{\ensuremath{\widehat{#1}}\xspace}}
\foreachLetter{1}{27}{\definehat}
\foreachletter{1}{27}{\definehat}
\def\definechk#1{\ea\gdef\csname #1chk\endcsname{\ensuremath{\widecheck{#1}}\xspace}}
\foreachLetter{1}{27}{\definechk}
\foreachletter{1}{27}{\definechk}
\def\defineul#1{\ea\gdef\csname u#1\endcsname{\ensuremath{\underline{#1}}\xspace}}
\foreachLetter{1}{27}{\defineul}
\foreachletter{1}{27}{\defineul}

\def\autofmt@n#1\autofmt@end{\mathrm{#1}}
\def\autofmt@b#1\autofmt@end{\mathbf{#1}}
\def\autofmt@d#1#2\autofmt@end{\mathbb{#1}\mathsf{#2}}
\def\autofmt@c#1#2\autofmt@end{\mathcal{#1}\mathit{#2}}
\def\autofmt@s#1#2\autofmt@end{\mathscr{#1}\mathit{#2}}
\def\autofmt@f#1\autofmt@end{\mathsf{#1}}
\def\autofmt@k#1\autofmt@end{\mathfrak{#1}}
\def\autofmt@u#1\autofmt@end{\underline{\smash{\mathsf{#1}}}}
\def\autofmt@U#1\autofmt@end{\underline{\underline{\smash{\mathsf{#1}}}}}
\def\autofmt@h#1\autofmt@end{\widehat{#1}}
\def\autofmt@r#1\autofmt@end{\overline{#1}}
\def\autofmt@t#1\autofmt@end{\widetilde{#1}}
\def\autofmt@k#1\autofmt@end{\check{#1}}

\def\auto@drop#1{}
\def\autodef#1{\ea\ea\ea\@autodef\ea\ea\ea#1\ea\auto@drop\string#1\autodef@end}
\def\@autodef#1#2#3\autodef@end{%
  \ea\def\ea#1\ea{\ea\ensuremath\ea{\csname autofmt@#2\endcsname#3\autofmt@end}\xspace}}
\def\autodefs@end{blarg!}
\def\autodefs#1{\@autodefs#1\autodefs@end}
\def\@autodefs#1{\ifx#1\autodefs@end%
  \def\autodefs@next{}%
  \else%
  \def\autodefs@next{\autodef#1\@autodefs}%
  \fi\autodefs@next}

\DeclareSymbolFont{bbold}{U}{bbold}{m}{n}
\DeclareSymbolFontAlphabet{\mathbbb}{bbold}

\mdef\delbar{\overline{\partial}}

\mdef\hf{\textstyle\frac12 }
\mdef\thrd{\textstyle\frac13 }
\mdef\qtr{\textstyle\frac14 }

\let\iso\cong

\mdef\Id{\mathrm{Id}}
\mdef\id{\mathrm{id}}
\alwaysmath{ell}
\alwaysmath{infty}

\alwaysmath{odot}
\def\frc#1/#2.{\frac{#1}{#2}}   % \frc x^2+1 / x^2-1 .
\mdef\ten{\mathrel{\otimes}}

\mdef\sqten{\mathrel{\boxtimes}}
\DeclareFontFamily{U}{mathb}{\hyphenchar\font45}
\DeclareFontShape{U}{mathb}{m}{n}{
      <5> <6> <7> <8> <9> <10> gen * mathb
      <10.95> mathb10 <12> <14.4> <17.28> <20.74> <24.88> mathb12
      }{}
\DeclareSymbolFont{mathb}{U}{mathb}{m}{n}
\DeclareFontSubstitution{U}{mathb}{m}{n}
\DeclareMathSymbol{\dotplus}       {2}{mathb}{"00}% name to be checked
\DeclareMathSymbol{\dotdiv}        {2}{mathb}{"01}% name to be checked
\DeclareMathSymbol{\dottimes}      {2}{mathb}{"02}% name to be checked
\DeclareMathSymbol{\divdot}        {2}{mathb}{"03}% name to be checked
\DeclareMathSymbol{\udot}          {2}{mathb}{"04}% name to be checked
\DeclareMathSymbol{\square}        {2}{mathb}{"05}% name to be checked
\DeclareMathSymbol{\Asterisk}      {2}{mathb}{"06}
\DeclareMathSymbol{\bigast}        {1}{mathb}{"06}
\DeclareMathSymbol{\coAsterisk}    {2}{mathb}{"07}
\DeclareMathSymbol{\bigcoast}      {1}{mathb}{"07}
\DeclareMathSymbol{\circplus}      {2}{mathb}{"08}% name to be checked
\DeclareMathSymbol{\pluscirc}      {2}{mathb}{"09}% name to be checked
\DeclareMathSymbol{\convolution}   {2}{mathb}{"0A}% name to be checked
\DeclareMathSymbol{\divideontimes} {2}{mathb}{"0B}% name to be checked
\DeclareMathSymbol{\blackdiamond}  {2}{mathb}{"0C}% name to be checked
\DeclareMathSymbol{\sqbullet}      {2}{mathb}{"0D}% name to be checked
\DeclareMathSymbol{\bigstar}       {2}{mathb}{"0E}
\DeclareMathSymbol{\bigvarstar}    {2}{mathb}{"0F}
\DeclareMathSymbol{\corresponds}   {3}{mathb}{"1D}% name to be checked
\DeclareMathSymbol{\updownarrows}          {3}{mathb}{"D6}
\DeclareMathSymbol{\downuparrows}          {3}{mathb}{"D7}
\DeclareMathSymbol{\Lsh}                   {3}{mathb}{"E8}
\DeclareMathSymbol{\Rsh}                   {3}{mathb}{"E9}
\DeclareMathSymbol{\dlsh}                  {3}{mathb}{"EA}
\DeclareMathSymbol{\drsh}                  {3}{mathb}{"EB}
\DeclareMathSymbol{\looparrowdownleft}     {3}{mathb}{"EE}
\DeclareMathSymbol{\looparrowdownright}    {3}{mathb}{"EF}
\DeclareMathSymbol{\curvearrowleftright}   {3}{mathb}{"F2}
\DeclareMathSymbol{\curvearrowbotleft}     {3}{mathb}{"F3}
\DeclareMathSymbol{\curvearrowbotright}    {3}{mathb}{"F4}
\DeclareMathSymbol{\curvearrowbotleftright}{3}{mathb}{"F5}
\DeclareMathSymbol{\leftsquigarrow}        {3}{mathb}{"F8}
\DeclareMathSymbol{\rightsquigarrow}       {3}{mathb}{"F9}
\DeclareMathSymbol{\leftrightsquigarrow}   {3}{mathb}{"FA}
\DeclareMathSymbol{\lefttorightarrow}      {3}{mathb}{"FC}
\DeclareMathSymbol{\righttoleftarrow}      {3}{mathb}{"FD}
\DeclareMathSymbol{\uptodownarrow}         {3}{mathb}{"FE}
\DeclareMathSymbol{\downtouparrow}         {3}{mathb}{"FF}
\DeclareMathSymbol{\varhash}       {0}{mathb}{"23}

\DeclareMathOperator\colim{colim}

\mdef\we{\overset{\sim}{\longrightarrow}}
\mdef\leftwe{\overset{\sim}{\longleftarrow}}

\let\fib\twoheadrightarrow

\def\acof{\mathrel{\mathrlap{\hspace{3pt}\raisebox{4pt}{$\scriptscriptstyle\sim$}}\mathord{\rightarrowtail}}}

\def\rightarrowtailfill@{\arrowfill@{\Yright\joinrel\relbar}\relbar\rightarrow}
\newcommand\xrightarrowtail[2][]{\ext@arrow 0055{\rightarrowtailfill@}{#1}{#2}}

\def\twoheadrightarrowfill@{\arrowfill@{\relbar\joinrel\relbar}\relbar\twoheadrightarrow}
\newcommand\xtwoheadrightarrow[2][]{\ext@arrow 0055{\twoheadrightarrowfill@}{#1}{#2}}

\def\slashedarrowfill@#1#2#3#4#5{%
  $\m@th\thickmuskip0mu\medmuskip\thickmuskip\thinmuskip\thickmuskip
   \relax#5#1\mkern-7mu%
   \cleaders\hbox{$#5\mkern-2mu#2\mkern-2mu$}\hfill
   \mathclap{#3}\mathclap{#2}%
   \cleaders\hbox{$#5\mkern-2mu#2\mkern-2mu$}\hfill
   \mkern-7mu#4$%
}
\def\rightslashedarrowfill@{%
  \slashedarrowfill@\relbar\relbar\mapstochar\rightarrow}
\newcommand\xslashedrightarrow[2][]{%
  \ext@arrow 0055{\rightslashedarrowfill@}{#1}{#2}}
\mdef\hto{\xslashedrightarrow{}}
\mdef\htoo{\xslashedrightarrow{\quad}}

\def\jd#1{\@jd#1\ej}
\def\@jd#1|-#2\ej{\@@jd#1,,\;\vdash\;\left(#2\right)}
\def\@@jd#1,{\@ifmtarg{#1}{\let\next=\relax}{\left(#1\right)\let\next=\@@@jd}\next}
\def\@@@jd#1,{\@ifmtarg{#1}{\let\next=\relax}{,\,\left(#1\right)\let\next=\@@@jd}\next}
\def\jdm#1{\@jdm#1\ej}
\def\@jdm#1|-#2\ej{\@@jd#1,,\\\vdash\;\left(#2\right)}
\long\def\my@drawfill#1#2;{%
\@skipfalse
\fill[#1,draw=none] #2;
\@skiptrue
\draw[#1,fill=none] #2;
}
\newif\if@skip
\newcommand{\skipit}[1]{\if@skip\else#1\fi}
\newcommand{\drawfill}[1][]{\my@drawfill{#1}}

\newcounter{nodemaker}
\newif\ifhyperref
\@ifpackageloaded{hyperref}{\hyperreftrue}{\hyperreffalse}
\iftac
  \let\your@state\state
  \def\state#1{\my@state#1}
  \def\my@state#1.{\gdef\currthmtype{#1}\your@state{#1.}}
  \let\your@staterm\staterm
  \def\staterm#1{\my@staterm#1}
  \def\my@staterm#1.{\gdef\currthmtype{#1}\your@staterm{#1.}}
  \let\@defthm\newtheorem
  \def\switchtotheoremrm{\let\@defthm\newtheoremrm}
  \def\defthm#1#2#3{\@defthm{#1}{#2}} % Ignore the third argument (for cleveref only)
  \let\your@section\section
  \def\section{\gdef\currthmtype{section}\your@section}
  \def\currthmtype{}
  \ifhyperref
    \def\autoref#1{\ref*{label@name@#1}~\ref{#1}}
  \else
    \def\autoref#1{\ref{label@name@#1}~\ref{#1}}
  \fi
  \AtBeginDocument{%
    \let\old@label\label%
    \def\label#1{%
      {\let\your@currentlabel\@currentlabel%
        \edef\@currentlabel{\currthmtype}%
        \old@label{label@name@#1}}%
      \old@label{#1}}
  }
  \let\cref\autoref
\else\ifcref
  \def\defthm#1#2#3{%
    \newaliascnt{#1}{thm}
    \newtheorem{#1}[#1]{#2}
    \aliascntresetthe{#1}
    \crefname{#1}{#2}{#3}% following brace must be on separate line to support poorman cleveref sed file
  }
\else
  \ifhyperref
    \def\defthm#1#2#3{% Ignore the third argument (for cleveref only)
      \newtheorem{#1}{#2}[section]%
      \expandafter\def\csname #1autorefname\endcsname{#2}%
      \expandafter\let\csname c@#1\endcsname\c@thm}
  \else
    \def\defthm#1#2#3{\newtheorem{#1}[thm]{#2}} % Ignore the third argument (for cleveref only)
    \ifx\SK@label\undefined\let\SK@label\label\fi
    \let\old@label\label
    \let\your@thm\@thm
    \def\@thm#1#2#3{\gdef\currthmtype{#3}\your@thm{#1}{#2}{#3}}
    \def\currthmtype{}
    \def\label#1{{\let\your@currentlabel\@currentlabel\def\@currentlabel%
        {\currthmtype~\your@currentlabel}%
        \SK@label{#1@}}\old@label{#1}}
    \def\autoref#1{\ref{#1@}}
  \fi
  \let\cref\autoref
\fi\fi

\newtheorem{thm}{Theorem}[section]
\ifcref
  \crefname{thm}{Theorem}{Theorems}
\else
  
\fi
\defthm{cor}{Corollary}{Corollaries}
\defthm{prop}{Proposition}{Propositions}
\defthm{lem}{Lemma}{Lemmas}
\defthm{sch}{Scholium}{Scholia}
\defthm{assume}{Assumption}{Assumptions}
\defthm{claim}{Claim}{Claims}
\defthm{conj}{Conjecture}{Conjectures}
\defthm{hyp}{Hypothesis}{Hypotheses}
\iftac\switchtotheoremrm\else\theoremstyle{definition}\fi
\defthm{defn}{Definition}{Definitions}
\defthm{notn}{Notation}{Notations}
\iftac\switchtotheoremrm\else\theoremstyle{remark}\fi
\defthm{rmk}{Remark}{Remarks}
\defthm{eg}{Example}{Examples}
\defthm{egs}{Examples}{Examples}
\defthm{ex}{Exercise}{Exercises}
\defthm{ceg}{Counterexample}{Counterexamples}

\ifcref
  \crefformat{section}{\S#2#1#3}
  \Crefformat{section}{Section~#2#1#3}
  \crefrangeformat{section}{\S\S#3#1#4--#5#2#6}
  \Crefrangeformat{section}{Sections~#3#1#4--#5#2#6}
  \crefmultiformat{section}{\S\S#2#1#3}{ and~#2#1#3}{, #2#1#3}{ and~#2#1#3}
  \Crefmultiformat{section}{Sections~#2#1#3}{ and~#2#1#3}{, #2#1#3}{ and~#2#1#3}
  \crefrangemultiformat{section}{\S\S#3#1#4--#5#2#6}{ and~#3#1#4--#5#2#6}{, #3#1#4--#5#2#6}{ and~#3#1#4--#5#2#6}
  \Crefrangemultiformat{section}{Sections~#3#1#4--#5#2#6}{ and~#3#1#4--#5#2#6}{, #3#1#4--#5#2#6}{ and~#3#1#4--#5#2#6}
  \crefformat{appendix}{Appendix~#2#1#3}
  \Crefformat{appendix}{Appendix~#2#1#3}
  \crefrangeformat{appendix}{Appendices~#3#1#4--#5#2#6}
  \Crefrangeformat{appendix}{Appendices~#3#1#4--#5#2#6}
  \crefmultiformat{appendix}{Appendices~#2#1#3}{ and~#2#1#3}{, #2#1#3}{ and~#2#1#3}
  \Crefmultiformat{appendix}{Appendices~#2#1#3}{ and~#2#1#3}{, #2#1#3}{ and~#2#1#3}
  \crefrangemultiformat{appendix}{Appendices~#3#1#4--#5#2#6}{ and~#3#1#4--#5#2#6}{, #3#1#4--#5#2#6}{ and~#3#1#4--#5#2#6}
  \Crefrangemultiformat{appendix}{Appendices~#3#1#4--#5#2#6}{ and~#3#1#4--#5#2#6}{, #3#1#4--#5#2#6}{ and~#3#1#4--#5#2#6}
  \crefformat{subappendix}{\S#2#1#3}
  \Crefformat{subappendix}{Section~#2#1#3}
  \crefrangeformat{subappendix}{\S\S#3#1#4--#5#2#6}
  \Crefrangeformat{subappendix}{Sections~#3#1#4--#5#2#6}
  \crefmultiformat{subappendix}{\S\S#2#1#3}{ and~#2#1#3}{, #2#1#3}{ and~#2#1#3}
  \Crefmultiformat{subappendix}{Sections~#2#1#3}{ and~#2#1#3}{, #2#1#3}{ and~#2#1#3}
  \crefrangemultiformat{subappendix}{\S\S#3#1#4--#5#2#6}{ and~#3#1#4--#5#2#6}{, #3#1#4--#5#2#6}{ and~#3#1#4--#5#2#6}
  \Crefrangemultiformat{subappendix}{Sections~#3#1#4--#5#2#6}{ and~#3#1#4--#5#2#6}{, #3#1#4--#5#2#6}{ and~#3#1#4--#5#2#6}
  \crefname{part}{Part}{Parts}
  \crefname{figure}{Figure}{Figures}
\fi

\iftac
  
  \let\your@endproof\endproof
  \def\my@endproof{\your@endproof}
  \def\endproof{\my@endproof\gdef\my@endproof{\your@endproof}}
  \def\qedhere{\tag*{\endproofbox}\gdef\my@endproof{\relax}}
\fi

\iftac
  \def\pr@@f[#1]{\subsubsection*{\sc #1.}}
\fi

\def\thmqedhere{\expandafter\csname\csname @currenvir\endcsname @qed\endcsname}

\ifbeamer\else

\fi

\ifbeamer\else
  \setitemize[1]{leftmargin=2em}
  \setenumerate[1]{leftmargin=*}
\fi

\iftac
  \let\c@equation\c@subsection
\else
  \let\c@equation\c@thm
\fi
\numberwithin{equation}{section}

\ifcref\else
  \@ifpackageloaded{mathtools}{\mathtoolsset{showonlyrefs,showmanualtags}}{}
\fi

\alwaysmath{alpha}
\alwaysmath{beta}
\alwaysmath{gamma}
\alwaysmath{Gamma}
\alwaysmath{delta}
\alwaysmath{Delta}
\alwaysmath{epsilon}
\mdef\ep{\varepsilon}
\alwaysmath{zeta}
\alwaysmath{eta}
\alwaysmath{theta}
\alwaysmath{Theta}
\alwaysmath{iota}
\alwaysmath{kappa}
\alwaysmath{lambda}
\alwaysmath{Lambda}
\alwaysmath{mu}
\alwaysmath{nu}
\alwaysmath{xi}
\alwaysmath{pi}
\alwaysmath{rho}
\alwaysmath{sigma}
\alwaysmath{Sigma}
\alwaysmath{tau}
\alwaysmath{upsilon}
\alwaysmath{Upsilon}
\alwaysmath{phi}
\alwaysmath{Pi}
\alwaysmath{Phi}
\mdef\ph{\varphi}
\alwaysmath{chi}
\alwaysmath{psi}
\alwaysmath{Psi}
\alwaysmath{omega}
\alwaysmath{Omega}
\let\al\alpha
\let\be\beta
\let\gm\gamma

\let\de\delta

\makeatother

\usepackage{hyperref}
\usepackage[all]{xy}
\usepackage{graphicx}
\usepackage{amsaddr}
\title{On the inadequacy of the projective structure with respect to the Univalence Axiom}
\author{Anthony Bordg}
\address{University of Cambridge, Department of Computer Science and Technology, William Gates Building, 15 JJ Thomson Ave, Cambridge CB3 0FD, UK \\ E-mail address: apdb3@cam.ac.uk \\ Phone: (+44)/(0) 1223 763741}
\thanks{This material is based upon work supported by the European Research Council Advanced Grant ALEXANDRIA (Project 742178) and by grant GA CR P201/12/G028 from the Czech Science Foundation.}

\newcommand{\ffive}[5]{\begin{array}{rrcl} #1: & #2&\longrightarrow &#3\\ &#4&\longmapsto &#5  \end{array}}
\newcommand{\ffour}[4]{\begin{array}{rcl} #1&\longrightarrow &#2\\ #3&\longmapsto &#4 \end{array}}
\newcommand{\fsix}[6]{\begin{array}{rcl} #1&\longrightarrow &#2\\ #3&\longmapsto &#4\\ #5&\longmapsto &#6\\\\ \end{array}} 

\newcommand{\pushoutcorner}[1][dr]{\save*!/#1+1.2pc/#1:(1,-1)@^{|-}\restore}

\newdir{ >}{{}*!/-5pt/\dir{>}}
\def\pcomd{\ar@{}[d]|{\circlearrowright}}
\def\ipcom{\ar@{}[rd]|{\circlearrowleft}}
\def\ipcomul{\ar@{}[lu]|{\circlearrowleft}}
\def\pcom{\ar@{}[rd]|{\circlearrowright}}
\def\pcomul{\ar@{}[lu]|{\circlearrowright}}
\def\pcomu{\ar@{}[ru]|{\circlearrowright}}  

\def\Gpd{\mathbf{Gpd}}
\def\GGpd{\mathbf{Gpd}^{C_2}}
\def\G{\mathbf{B}(C_2)}
\def\Z2{C_2}
\def\f{_{\mathbf{f}}}
\def\0{\mathbf{0}}
\def\1{\mathbf{1}}

\mdef\el{\mathsf{El}}

\begin{document}

\begin{abstract}
	In this article the author endows the functor category $[\mathbf{B}(C_2),\Gpd]$ with the structure of a type-theoretic fibration category with a universe using the projective fibrations. It offers a new model of Martin-L\"of type theory with dependent sums, dependent products, identity types and a universe. It turns out that this universe, the natural candidate that lifts the univalent universe of small discrete groupoids in the groupoid model of Hofmann and Streicher, is not univalent. 
\end{abstract}

\maketitle

\section{Introduction}
\label{sec:introduction}

In the seventies Per Martin-L\"of set a framework out, suitable for constructive mathematics, called \emph{Martin-L\"of Type Theory} (MLTT for short). It is well known that MLTT enjoys very nice computational properties that make it suitable for the formalization of mathematics with a proof assistant. Recently Vladimir Voevodsky added an axiom to MLTT, the so-called \emph{Univalence Axiom} (UA for short). Given a type-theoretic universe, UA roughly asserts an equivalence between the identity type of any two small types (\textit{i.e.} two elements of the universe) and the type of weak equivalences between them. This new framework, MLTT together with UA, was coined \emph{Univalent Foundations} (UF for short). \\
Voevodsky found an interpretation of UF in the category of simplicial sets using Kan simplicial sets, where the universe is interpreted as the base of a universal Kan fibration (\textit{cf.} \cite{klv:ssetmodel} for details). Through the notion of a \emph{type-theoretic fibration category}, models of UF were later pursued by Michael Shulman \cite{shulman:invdia, shulman:elreedy, shulman:eiuniv}. The line of research initiated by Michael Shulman consists in the exploration of the stability of UA, in particular in the following sense : given a type-theoretic fibration category $\mathscr{C}$ together with a univalent universe, one wants to lift this type-theoretic fibration category with its univalent universe to the functor category $[\mathcal{D},\mathscr{C}]$, where the index category $\mathcal{D}$ is a small category. This goal was achieved in some specific cases.\\
First, in \cite{shulman:invdia} Shulman succeeded when $\mathcal{D}$ is an \emph{inverse category} by using the so-called \emph{Reedy model structure} on the functor category. Second, in \cite{shulman:elreedy} Shulman succeeded with the same model structure when $\mathscr{C}$ is the category $\mathbf{sSet}$ of simplicial sets and $\mathcal{D}$ is any \emph{elegant Reedy category}. Note that inverse categories are particular cases of elegant Reedy categories that are themselves particular cases of (strict) \emph{Reedy categories}. Since Reedy categories do not allow non-trivial isomorphisms, this kind of index category has severe constraints. The difficulty in handling non-trivial isomorphisms in the index category seems a challenge to the usefulness of the Reedy model structure with respect to the stability of UA. Around the same time Shulman in \cite{shulman:eiuniv} and the author in his PhD thesis \cite{bordg:thesis} tried different alternative model structures. \\
In particular, the author explored the possibility of using the so-called \emph{projective model structure} to endow a functor category with the structure of a type-theoretic fibration category with a univalent universe. The projective model structure is an attractive candidate, since fibrations, the class of maps interpreting dependent types in a type-theoretic fibration category, are simply defined objectwise. Starting from the groupoid model \cite{hs:gpd-typethy} of Hofmann and Streicher with its univalent universe of small discrete groupoids, we treated the 2-dimensional case where $\mathscr{C}$ is the category $\Gpd$ of groupoids and $\mathcal{D}$ is $\mathbf{B}(C_2)$, namely the groupoid associated with the group with two elements that presents in this context the interesting technical challenge of containing a non-trivial automorphism. We discovered that the projective fibrations allow one to endow the functor category $[\mathbf{B}(C_2), \Gpd]$ with the structure of a type-theoretic fibration category with a universe. But, while this universe is (according to \ref{lem:4.4}) the natural universe that lifts, with respect to the projective setting, the univalent universe of small discrete groupoids in $\Gpd$, it turns out this universe is not univalent. Even a weaker form of UA, namely function extensionality, does not hold in this new type-theoretic fibration category.

\subsection*{Acknowledgments}

I would like to thank Andr\'e Hirschowitz, Peter LeFanu Lumsdaine, Michael Shulman for helpful discussions and the anonymous referees for many helpful suggestions.

\section{The projective model structure on $\GGpd$ made explicit}
\label{sec:pms}

We will denote the functor category $[\mathbf{B}(C_2), \Gpd]$ simply by $\GGpd$. The reader should note that an object in $\GGpd$ is nothing but a groupoid equipped with an involution, and a morphism in $\GGpd$ is nothing but an equivariant functor, namely a functor between groupoids compatible with the involutions on the domain and codomain. Such a groupoid will be denoted by a capital letter, $A$ for instance, and the corresponding Greek letter $\al$ will be used to refer to its involution (except when stated otherwise). \\
We recall that for the natural model structure on $\Gpd$ \cite{rezk:folk, strickland:localduality}, the weak equivalences are the equivalences of categories, the fibrations are the isofibrations, and a functor is a cofibration if it is injective on objects.
Since the natural model structure on $\Gpd$ is cofibrantly generated and $\mathbf{B}(C_2)$ is a small category, there exists the projective model structure \cite[proposition A.2.8.2]{lurie:higher-topoi} on $\GGpd$. Hereinafter by an objectwise weak equivalence (\textit{resp.} an objectwise fibration) one means a map whose underlying map of groupoids is a weak equivalence (\textit{resp.} a fibration) in $\Gpd$.\\
Recall that one can describe this projective model structure by:
\begin{itemize} 
	\item Weak equivalences are the objectwise weak equivalences.
	\item Fibrations are the objectwise fibrations.
	\item Cofibrations are those maps with the left lifting property with respect to acyclic fibrations (fibrations which are simultaneously weak equivalences).
\end{itemize}

\begin{notn}
\label{notn:2.1}	
As usual in category theory, the initial object and the terminal object of $\Gpd$ will be denoted by $\0$ and $\1$ respectively. We will use the letter $\bI$ for the groupoid with two distinct points and one isomorphism $\phi$ between them. We denote by $i$ the inclusion $i: \1 \hookrightarrow \bI$. \\
We have an obvious functor from $\G$ to $\1$ and an obvious inclusion from $\1$ to $\G$. These two functors induce by precomposition the two following functors,
$$\ffive{\_}{\GGpd}{\Gpd}{G}{\underline{G}}$$
namely the forgetful functor that maps a groupoid $G$ equipped with an involution to its underlying groupoid;
$$\ffive{!}{\Gpd}{\GGpd}{G}{G!} $$
that maps a groupoid to the same groupoid together with the identity involution. \\
The forgetful functor has a left adjoint denoted $S$ that maps a groupoid $G$ to $S(G)\coloneqq G\textstyle\coprod G$ together with the involution that swaps the two copies of $G$. \\
The functor $!$ has a right adjoint, namely the  fixed-points functor, 
$$\ffive{()^{\Z2}}{\GGpd}{\Gpd}{G}{G^{\Z2}}$$ where $G^{\Z2}$ is the subgroupoid of $G$ of fixed points and fixed morphisms under the $\Z2$-action. Note that $G^{\Z2}$ is $\lim G$.\\
Since limits and colimits are pointwise in a presheaf category, $\0!$ and $\1!$ (shortened to $\0$ and $\1$ when no confusion is possible) are the corresponding initial and terminal objects in $\GGpd$. \\
Given a groupoid $G$ together with an involution, we will denote by $G\f$ the full subgroupoid of $G$ consisting of its fixed points, {\em i.e.} the objects are the fixed points, but we take all (not necessarily fixed) morphisms between them. \\
Last, given a morphism of groupoids $f:A\rightarrow B$ and $x\in B$ , $f^{-1}\lbrace{x}\rbrace$ will denote the subgroupoid of $A$ whose objects are objects of $A$ above $x$ and morphisms are morphisms of $A$ above the identity $1_x$.  
\end{notn}

Knowing the generating acyclic cofibrations in $\Gpd$, by looking at the construction of the projective model structure one finds the generating acyclic cofibrations with respect to the projective model structure on $\GGpd$ \cite[proof of proposition A.2.8.2]{lurie:higher-topoi}. Indeed, a set (actually it is a singleton in that case) of generating acyclic cofibrations is given by the following inclusion: $$S(i): S(\1)\hookrightarrow S(\bI)\;.$$

\begin{prop}
\label{prop:2.2}
	Let $f:A\rightarrow B$ be a morphism in $\GGpd$. The following are equivalent: 
	\begin{enumerate}[label=(\roman*)]
		\item $f$ is an acyclic cofibration.
		\item $\underline{f}$ is an acyclic cofibration and induces a bijection between the set of fixed points of $A$ and the set of fixed points of $B$.
		\item $\underline{f}$ is an acyclic cofibration and induces an isomorphism between $A^{\Z2}$ and $B^{\Z2}$.
		\item $\underline{f}$ is an acyclic cofibration and induces an isomorphism between $A\f$ and $B\f$. 
	\end{enumerate}
\end{prop}
\begin{proof}
\label{proof:prop:2.2}
	We prove successively $(i)\Rightarrow (ii)$, $(ii)\Rightarrow (iii)$, $(iii)\Rightarrow (iv)$ and $(iv)\Rightarrow (i)$.
	\begin{itemize}
		\item $(i)\Rightarrow (ii)\colon$ assume that $f$ is an acyclic cofibration. It is well-known that $\underline{f}$ is an acyclic cofibration \cite[proposition 11.6.2]{hirschhorn:modelcats}. Moreover, note that if $x$ is a fixed point of $A$, then one has $f(x) = f(\al(x)) = \be(f(x))$. Hence, $f(x)$ is a fixed point of $B$. We also know that $f$ is injective on objects as a cofibration between groupoids. We need to prove that any fixed point in $B$ is the image of a fixed point in $A$. To achieve this, the reader can check this fact for the generating acyclic cofibration $S(i)$ and the stability of this fact under pushouts, transfinite compositions and retractions. One concludes that $f$ induces a bijection between the fixed points in $A$ and the fixed points in $B$.
		\item $(ii)\Rightarrow (iii)\colon$ it is a straighforward consequence of $f$ being fully faithful.
		\item $(iii)\Rightarrow (iv)\colon$ \textit{idem}.
		\item $(iv)\Rightarrow (i)\colon$ Since $\underline{f}$ is an acyclic cofibration of groupoids where $A\f$ and $B\f$ are isomorphic, $\underline{f}$ is isomorphic to the inclusion of a full subgroupoid of $B$ equivalent to $B$ where $A\f$ and $B\f$ are equal. Let $((\text{Ob}B\setminus\text{Ob}A)/\Z2, \leq)$ be the set of orbits of $(\text{Ob}B\setminus\text{Ob}A)$ under the $\Z2$-action together with a well-ordering. Let $\lambda$ be the order type of this well-ordering and $g\colon (\text{Ob}B\setminus\text{Ob}A)/\Z2 \rightarrow \lambda$ an order-preserving bijection. We will construct a $\lambda$-sequence $X$ of pushouts of $S(i)$, where we add the elements of $(\text{Ob}B\setminus\text{Ob}A)$ to $A$ by following our well-ordering. Take $X_0 \coloneqq A$. For $\gm$ such that $\gm + 1 <\lambda$, $X_{\gm + 1}$ is defined as the following pushout. Let $s$ be the element of $(\text{Ob}B\setminus\text{Ob}A)/\Z2$ that corresponds to $\gm + 1$ under $g$. Actually, $s$ is a set with two distinct elements $\lbrace x,\be(x)\rbrace$. Since $f$ is essentially surjective, there exists an isomorphism $\varphi\colon y\rightarrow x$ with $y\in A$, and we make the following pushout
		$$\xymatrix{S(\1) \ar[r]^l\ar@{^{(}->}[d]_{S(i)} & X_{\gm} \ar[d] \\
			S(\bI) \ar[r] & X_{\gm + 1}\;,\pushoutcorner}$$
		where $l$ maps the two objects of $S(1)$ to $y$ and $\al(y)$, respectively. Last, if $\gm < \lambda$ is a limit ordinal, then $X_{\gm}$ is $\underset{\delta < \gm}{\colim}\,X_{\delta}$. For every $\gm <\lambda$, $X_{\gm}$ is a full subgroupoid of $B$ stable under the involution $\be$, and $f$ is the transfinite composition of the $\lambda$-sequence $X$. So, $f$ is an acyclic cofibration.
	\end{itemize}
\end{proof}

\begin{prop}
\label{prop:2.3}
	Let $G$ be a groupoid equipped with an involution and $G'$ a subgroupoid of $G$ stable under that involution such that $G'\f = G\f$ and the inclusion map from $G'$ to $G$ is an equivalence of groupoids. Then the inclusion map is an acyclic cofibration.
\end{prop}
\begin{proof}
\label{proof:prop:2.3}	
	Since by assumption the inclusion is an objectwise acyclic cofibration and $G'\f = G\f$, it is straighforward by \ref{prop:2.2}.
\end{proof}

\section{$\GGpd$ as a type-theoretic fibration category}
\label{sec:ttfc}

We recall below the definition of a type-theoretic fibration category \cite[Definition 7.1]{shulman:eiuniv}.

\begin{defn}
\label{def:ttfc}
	A \textbf{type-theoretic fibration category} is a category $\mathscr{C}$ with:
	\begin{enumerate}[leftmargin=*,label=(\arabic*)]
		\item A terminal object 1.\label{item:cat1}
		\item A subcategory of \textbf{fibrations} containing all the isomorphisms and all the morphisms with codomain 1. A morphism is called an \textbf{acyclic cofibration} if it has the left lifting property with respect to all fibrations.\label{item:cat2}
		\item All pullbacks of fibrations exist and are fibrations.\label{item:cat3}
		\item For every fibration $g: A\rightarrow B$,
		the pullback functor $g^*: \mathscr{C}/B \to \mathscr{C}/A$ has a partial right adjoint $\Pi_g$, defined at all fibrations over $A$, and whose values are fibrations over $B$.
		This implies that acyclic cofibrations are stable under pullback along $g$.\label{item:cat4}
		\item Every morphism factors as an acyclic cofibration followed by a fibration.\label{item:cat5}
	\end{enumerate}
\end{defn}

\begin{rmk}
\label{rmk:modeltt}	
	A type-theoretic fibration category corresponds to the categorical structure necessary for interpreting a type theory with a unit type, dependent sums, dependent products, and intensional identity types.
\end{rmk}	

\begin{notn}
\label{notn:3.2}
In a type-theoretic fibration category we denote a fibration by a two-headed arrow $\fib$ and an acyclic cofibration by $\acof$.
\end{notn}	

\begin{lem}
\label{lem:rightpropernessGpd}
	In the natural model structure on $\Gpd$ acyclic cofibrations are stable by pullback along any fibration.
\end{lem}
\begin{proof}
\label{proof:lem:rightpropernessGpd}
All objects are fibrant, hence the natural model structure is right proper. Moreover, since the cofibrations are the functors that are injective on objects, they are stable under pullback along any morphism.	
\end{proof}

We prove the analogous result for $\GGpd$ with respect to the projective model structure.

\begin{lem}
\label{lem:rightpropernessGGpd}
	In the projective model structure on $\GGpd$ acyclic cofibrations are stable under pullback along any fibration.
\end{lem}
\begin{proof}
\label{proof:lem:rightpropernessGGpd}	
This is an immediate consequence from \ref{lem:rightpropernessGpd} and \ref{prop:2.2} $(ii)$, since pullbacks and fixed points commute (both being limits).
\end{proof}

\begin{lem}
\label{lem:pullbackfunctor}
	The pullback functor along a fibration preserves acyclic cofibrations with respect to the projective model structure on $\GGpd$.
\end{lem}
\begin{proof}
\label{proof:lem:pullbackfunctor}	
	This follows directly from \ref{lem:rightpropernessGGpd}.
\end{proof}

\begin{thm}
\label{thm:rightadjoint}
	For every fibration $g\colon A \twoheadrightarrow B$ in $\GGpd$, the pullback functor 
	$$g^* \colon \GGpd/B \rightarrow \GGpd/A$$ 
	has a right adjoint $\Pi_g$, and $\Pi_g$ maps fibrations over $A$ to fibrations over $B$.
\end{thm}
\begin{proof}
\label{proof:thm:rightadjoint}
Since right adjoints preserve limits, the following square commutes up to isomorphism
	$$\xymatrix{\GGpd/B \ar[r]^{g^*} \ar[d]_{(\_)} & \GGpd/A \ar[d]^{(\_)} \\ \Gpd/\underline{B} \ar[r]_{(\underline{g})^*} & \Gpd/\underline{A}\;.}$$
	Moreover, the vertical functors preserve and reflect colimits, and $(\underline{g})^*$ is a left adjoint \cite[lemma 4.3, theorem 4.4]{giraud:meth-desc} hence it preserves colimits too. Thus $g^*$ preserves colimits. Given that $\GGpd/B$ and $\GGpd/A$ are locally presentable categories, we conclude that $g^*$ has a right adjoint. \\
	By adjointness and \ref{lem:pullbackfunctor}, $\Pi_g$ preserves fibrations.
\end{proof}

\begin{rmk}
\label{rmk:giraud}
When the involutions involved in the statement of theorem \ref{thm:rightadjoint} are identities, we recover Giraud's theorem \cite[lemma 4.3, theorem 4.4]{giraud:meth-desc}.	
\end{rmk}	

\begin{cor}
\label{cor:projectivettfc}
	The category $\GGpd$ has the structure of a type-theoretic fibration category with respect to projective fibrations.
\end{cor}
\begin{proof}
\label{proof:cor:projectivettfc}	
	The required conditions \ref{item:cat1}, \ref{item:cat2}, \ref{item:cat3} and \ref{item:cat5} are straighforward. The theorem \ref{thm:rightadjoint} allows us to conclude that \ref{item:cat4} holds.
\end{proof}

\section{A universe in $\GGpd$}
\label{sec:universes}

We recall the notion of a \emph{universe} \cite[Definition 6.12]{shulman:invdia} in a type-theoretic fibration category.

\begin{defn}
\label{def:universe}
	A fibration $p: \widetilde{U}\fib U$ in a type-theoretic fibration category $\mathscr{C}$ is a \textbf{universe} if the following hold.
	\begin{enumerate}
		\item Pullbacks of $p$ are closed under composition and contain the identities.\label{item:u1}
		\item If $f: B\fib A$ and $g: A\fib C$ are pullbacks of $p$, so is $\Pi_g f \fib C$.\label{item:u2}
		\item If $A\fib C$ and $B\fib C$ are pullbacks of $p$, then any morphism $f: A\to B$ over $C$ factors as an acyclic cofibration followed by a pullback of $p$.\label{item:u3}
	\end{enumerate}
\end{defn}

\begin{defn}
\label{defn:smallfib}
	Given a universe $p\colon\widetilde{U}\rightarrow U$ in a type-theoretic fibration category, a \textbf{small fibration}, or a $U$-small fibration, is a pullback of $p$. 	
\end{defn}

\begin{rmk}
\label{rmk:modelttwithuniv}
 A universe in a type-theoretic fibration category interprets a universe type in type theory.	
\end{rmk}	 

We now move on to constructing universes in the type-theoretic fibration category on $\GGpd$ given in \ref{cor:projectivettfc}. Note that the groupoid model \cite{hs:gpd-typethy} of type theory can be reformulated \cite[Examples 2.16]{shulman:invdia} in terms of a type-theoretic fibration category using the natural model structure on $\Gpd$. Given any inaccessible cardinal $\kappa$, in this type-theoretic fibration structure on $\Gpd$ there exists a (univalent) universe $p: \widetilde{V_\kappa}\rightarrow V_\kappa$, where $V_\kappa$ is the groupoid whose objects are $\kappa$-small discrete groupoids with isomorphisms between them, $\widetilde{V_\kappa}$ the corresponding groupoid of pointed discrete groupoids and $p$ the obvious projection. The $\kappa$-smallness means that the set of objects of a discrete groupoid has cardinality strictly less than $\kappa$. The $V_\kappa$-small fibrations are precisely the discrete fibrations of groupoids with $\kappa$-small fibers. So, projective fibrations being objectwise fibrations, a natural candidate for a (univalent) universe in $\GGpd$ would be a universal fibration that classifies projective fibrations that are objectwise discrete fibrations of groupoids with $\kappa$-small fibers. \\
Below we define $\widetilde{U}, U$ and $p\colon \widetilde{U}\rightarrow U$ in $\GGpd$. For the rest of this section $\kappa$ is an inaccessible cardinal.
\begin{itemize}
	\item The objects of the groupoid $\widetilde{U}$ are dependent tuples of the form $(A, B, \varphi, a)$, where $A, B$ are $\kappa$-small discrete groupoids, $\varphi\colon A \rightarrow B$ is an isomorphism in $\Gpd$, and $a$ is an object of $A$.
	\item The morphisms in $\widetilde{U}$ between $(A, B, \varphi, a)$ and $(C, D, \psi, c)$ are pairs of isomorphisms of the form $(\rho\colon A\rightarrow C, \tau\colon B\rightarrow D)$ such that $\psi\circ \rho = \tau\circ \varphi$ and $\rho (a) = c$. 
\end{itemize}
The composition in $\widetilde{U}$ is given by 
$$(\rho',\tau')\circ(\rho,\tau)\coloneqq(\rho'\circ\rho,\tau'\circ\tau)\;.$$ Note that $\widetilde{U}$ is a groupoid. Indeed, the inverse of the morphism $(\rho,\tau)$ is given by
$$(\rho,\tau)^{-1}\coloneqq(\rho^{-1},\tau^{-1})\;.$$ We equip $\widetilde{U}$ with the involution $\tilde{\upsilon}$ as follows,
$$\fsix{\tilde{\upsilon}\colon\widetilde{U}}{\widetilde{U}}{(A,B,\varphi,a)}{(B,A,\varphi^{-1},\varphi(a))}{(\rho,\tau)}{(\tau,\rho)\;.}$$ One denotes by $U$ the ``unpointed'' version of $\widetilde{U}$, \textit{i.e.} objects are of the form $(A,B,\varphi)$ and morphisms of the form $(\rho,\tau)$, with its corresponding involution $\upsilon$. We define the morphism $p$ in $\GGpd$ as the projection 
$$\fsix{p\colon\widetilde{U}}{U}{(A,B,\varphi,a)}{(A,B,\varphi)}{(\rho,\tau)}{(\rho,\tau)\;.}$$
We want to prove that $p:\widetilde{U}\rightarrow U$ is a universe in the type-theoretic fibration category \ref{cor:projectivettfc}.

\begin{defn}
\label{defn:discretefib}
	In the natural model structure on $\Gpd$, a \textbf{discrete fibration} of groupoids is a fibration satisfying the property that given any isomorphism $\varphi$ in the target groupoid and any object $x$ in the fiber of $\text{dom}(\varphi)$, there exists a \emph{unique} lift of $\varphi$ at $x$ in the domain groupoid. \\ 
	The map that sends any such lifting problem to its unique solution is called a (split) \emph{cleavage} of $f$.
\end{defn}

\begin{lem}
\label{lem:4.3}
	The morphism $p:\widetilde{U} \rightarrow U$ is a projective fibration in $\GGpd$ between fibrant objects, whose underlying morphism of groupoids $\underline{p}$ is a discrete fibration.
\end{lem}
\begin{proof}
\label{proof:lem:4.3}
	The projective fibrations being objectwise, the terminal object being pointwise and every groupoid being fibrant with respect to the natural model structure on $\Gpd$, every object in $\GGpd$ is fibrant with respect to the projective model structure. In particular, the groupoids $\widetilde{U}$ and $U$ are fibrant objects. Moreover, $p$ is an objectwise discrete fibration and we define its unique split cleavage $c_p$ as follows. Given $(\rho,\tau)$ an isomorphism in $U$ 
	and $(\text{dom}(\rho,\tau),x)$ an element in the $p$-fiber of $\text{dom}(\rho,\tau)$, we have no choice but to take $c_{p}((\rho,\tau),(\text{dom}(\rho,\tau),x))\coloneqq (\rho,\tau)$ seen as a morphism in $\widetilde{U}$ between $(\text{dom}(\rho,\tau),x)$ and $(\text{cod}(\rho,\tau),\rho(x))$. 
\end{proof}

\begin{lem}
\label{lem:4.4} 
	The $U$-small fibrations in $\GGpd$ are precisely the fibrations whose underlying morphisms of groupoids are discrete fibrations with $\kappa$-small fibers.
\end{lem}
\begin{proof}
\label{proof:lem:4.4}
	Since discrete fibrations are defined by a unique-lifting property, they are stable under pullback. So, the first direction is clear. \\
	Conversely, given $f: A\rightarrow B$, assume that $\underline{f}$ is a discrete fibration of groupoids with $\kappa$-small fibers. We denote $c_{f}$ its cleavage. One has to display $f$ as a pullback of $p$ in $\GGpd$ along a morphism $g$. We define $g$ as follows
	$$\fsix{g\colon B}{U}{x}{(f^{-1}\lbrace x\rbrace,f^{-1}\lbrace \be(x)\rbrace,\alpha_{x})}{x\xrightarrow{\sigma}y}{(\rho_{\sigma},\tau_{\sigma})\;.}$$ Since $f$ is a discrete fibration with $\kappa$-small fibers, these groupoids are discrete and $\kappa$-small. Moreover, for $x\in B$ we define $\alpha_{x}$ as the isomorphism obtained from the restriction of $\alpha$ to $f^{-1}\lbrace x\rbrace$. 
	Given $\sigma\colon x\rightarrow y$ in $B$, we define $\rho_{\sigma}$ as follows 
	$$\ffour{\rho_{\sigma}\colon f^{-1}\lbrace x\rbrace}{f^{-1}\lbrace y\rbrace}{z}{\text{cod}(c_{f}(\sigma,z))\;.}$$
	In the same way one has 
	$$\ffour{\tau_{\sigma}\colon f^{-1}\lbrace \be(x)\rbrace}{f^{-1}\lbrace \be(y)\rbrace}{z}{\text{cod}(c_{f}(\be(\sigma),z))\;.}$$
	 The reader can easily check that $\rho_{\sigma}$ and $\tau_{\sigma}$ are isomorphisms, that $\al_y\circ \rho_{\sigma}$ is equal to $\tau_{\sigma}\circ \al_{x}$, and $g$ is functorial and equivariant. It remains to check that $A$ is isomorphic to $B\times_{U}\widetilde{U}$ above $B$, \textit{i.e.} we need to provide an isomorphism $\chi\colon A\rightarrow B\times_{U}\widetilde{U}$ such that $pr_{1}\circ\chi = f$, where $pr_{1}\colon B\times_{U}\widetilde{U}\rightarrow B$ is the first projection. Let define $\chi$ as follows
	$$\fsix{\chi\colon A}{B\times_{U}\widetilde{U}}{x}{(f(x),f^{-1}\lbrace f(x)\rbrace,f^{-1}\lbrace \be(f(x))\rbrace,\alpha_{f(x)},x)}{x\xrightarrow{\sigma}y}{(f(\sigma),\rho_{f(\sigma)},\tau_{f(\sigma)})\;.}$$
	The functor $\chi$ is equivariant and it is actually an isomorphism with $\chi^{-1}$ given by 
	$$\fsix{\chi^{-1}\colon B\times_{U}\widetilde{U}}{A}{(x,f^{-1}\lbrace x\rbrace,f^{-1}\lbrace \be(x)\rbrace,\al_{x},z)}{z}{(\sigma,\rho_{\sigma},\tau_{\sigma})}{c_{f}(\sigma,\_)\;,}$$
	where $\_$ denotes the last element of the tuple $\text{dom}(\rho_{\sigma},\tau_{\sigma})$.
\end{proof}

\begin{rmk}
\label{rmk:4.5}
	The previous lemma expresses in which sense our putative universe $p\colon \widetilde{U}\rightarrow U$ in $\GGpd$ is the natural candidate with respect to the projective model structure that lifts the (univalent) universe $p\colon \widetilde{V_\kappa}\rightarrow V_{\kappa}$ in $\Gpd$.
\end{rmk}

\begin{lem}
\label{lem:u1}
	In $\GGpd$ $\kappa$-small fibrations are closed under composition and contain the identities.
\end{lem}
\begin{proof}
\label{proof:lem:u1}
	Since according to \ref{lem:4.4} small fibrations are the objectwise discrete fibrations of groupoids with small fibers, it is straightforward.
\end{proof}

\begin{lem}
\label{lem:u2}
	If $f$ and $g$ are $\kappa$-small fibrations in $\GGpd$, so is $\Pi_{g}f$.
\end{lem}
\begin{proof}
\label{proof:lem:u2}
	It suffices to prove that $\underline{\Pi_gf}$ is a $V_{\kappa}$-small fibration of groupoids. But, according to the explicit description given in \cite[lemma 4.3.4]{bordg:thesis} $\underline{\Pi_gf}$ is $\Pi_{\underline{g}}\underline{f}$\footnote{Right adjoints to pullback functors are not generally computed pointwise, but this is true in our case.}. Since $\underline{g}, \underline{f}$ are $V_{\kappa}$-small fibrations by assumption and $V_{\kappa}$ is a universe in $\Gpd$, $\underline{\Pi_gf}$ is a $V_{\kappa}$-small fibration between groupoids. 
\end{proof}

\begin{lem}
\label{lem:diagonalmap}
	For any $U$-small fibration $f$, the diagonal map $\Delta_f$ is a $U$-small fibration.
\end{lem}
\begin{proof}
\label{proof:lem:diagonalmap}
	Since pullbacks are pointwise in $\GGpd$, one has $\underline{\Delta_f} = \Delta_{\underline{f}}$. It suffices to prove that $\Delta_{\underline{f}}$ is a $V_{\kappa}$-small fibration of groupoids, namely a discrete fibration with small fibers. A lifting problem for $\Delta_{\underline{f}}$ with respect to the generating acyclic cofibration $i$ is nothing but a pair of isomorphisms $(\varphi,\psi)$ in $E^2$ such that $f(\varphi) = f(\psi)$ and $\text{dom}(\varphi) = \text{dom}(\psi)$. Since $\underline{f}$ is a discrete fibration, one has $\varphi = \psi$. So, $\Delta_{\underline{f}}$ is a discrete fibration and its fibers are obviously small because any fiber is either empty or a singleton.    
\end{proof}

\begin{thm}
\label{thm:universe}
	The morphism $p\colon\widetilde{U}\rightarrow U$ is a universe in the type-theoretic fibration category $\GGpd$.
\end{thm}
\begin{proof}
\label{proof:thm:universe}
 The lemmas \ref{lem:u1} and \ref{lem:u2} take care of \ref{item:u1} and \ref{item:u2} respectively. According to \cite[Remark 6.13]{shulman:invdia}, \ref{item:u3} is equivalent (under \ref{item:u1}) to the fact that any small fibration has a small path fibration. But, by \ref{lem:diagonalmap} any small fibration $f$ has indeed a small path fibration.	
\end{proof}

Now to go further, we need to recall what it means for a universe in a type-theoretic fibration category to be \emph{univalent}\label{univalenceproperty} (see also \cite[section 7]{shulman:invdia}). Let $\mathsf{Type}$ be a universe in the type theory under consideration. Given two \emph{small} types, \textit{i.e.} two elements of $\mathsf{Type}$, there is the type of weak equivalences between them. In a type-theoretic fibration category with a universe, this dependent type is represented by a fibration $E \fib U\times U$. Moreover, there is a natural map $U \rightarrow E$ that sends a type to its identity equivalence. By \ref{item:cat5} one can factor the diagonal map $\de\colon U \rightarrow U\times U$ as an acyclic cofibration followed by a fibration in the following commutative diagram
$$\xymatrix{U \ar[r]\ar@{ >->}[d]^{\rotatebox[origin=c]{90}{$\sim$}} & E \ar@{->>}[d] \\ PU \ar@{->>}[r] \ar@{-->}[ru] & U\times U\;.}$$
The universe $p\colon \widetilde{U} \fib U$ is univalent if the map $U \rightarrow E$ over $U\times U$ is a right homotopy equivalence, or equivalently (by the \emph{2-out-of-3} property and the fact that $U$ is fibrant like any object of a type-theoretic fibration category) if the dashed map is a right homotopy equivalence.

\begin{rmk}
Since in general there is no class of cofibrations in a type-theoretic fibration category, only a class of acyclic cofibrations, the definitions of right homotopies and right homotopy equivalences in this context involve only the so-called very good path objects.	
\end{rmk}
	
\section{Right homotopy equivalences in $\GGpd$}
\label{sec:righthomeq}

Now, we develop a few basic facts about right homotopy equivalences, then we give an explicit characterization of right homotopy equivalences with respect to the projective structure on $\GGpd$.

\begin{prop}
\label{prop:5.1}
	If $\mathscr{C}$ is a type-theoretic fibration category and $f\colon A \acof B$ is an acyclic cofibration, then $f$ is a right homotopy equivalence.
\end{prop}
\begin{proof}
\label{proof:prop:5.1}
See \cite[lemma 3.6]{shulman:invdia}.
\end{proof}

\begin{rmk}
\label{rmk:5.2}
	In particular, if $f$ is an acyclic cofibration in the type-theoretic fibration structure on $\GGpd$ given in \ref{cor:projectivettfc}, then $f$ is a right homotopy equivalence.
\end{rmk}

\begin{rmk}
\label{rmk:5.3}
	Since not all objects are cofibrant in the projective model structure on $\GGpd$, right homotopy equivalences are not the same as the weak equivalences of the model structure.
\end{rmk}

\begin{defn}
\label{defn:weaklyconnected}	
	Let $A$ be a groupoid together with an involution $\al$. One says that $A$ is \textbf{weakly connected} if and only if for every pair $(x,y)$ in $\text{Ob}(A)^2$ either $x$ and $y$ are in the same connected component of $A$ or $x$ and $\al(y)$ are in the same connected component.
\end{defn}

\begin{lem}
\label{lem:weaklyconnected}
	Every groupoid together with an involution is a coproduct in $\GGpd$ of weakly connected groupoids with involutions.
\end{lem}
\begin{proof}
\label{proof:lem:weaklyconnected}
	Let $A$ be a groupoid together with an involution $\al$. Given $x$ in $\text{Ob}(A)$, we denote by $A_x$ the connected component of $x$ in the groupoid $A$. Now, we denote by $A_x^{\Z2}$ the full subgroupoid of $A$ whose set of objects is $\text{Ob}(A_x)\bigcup\text{Ob}(A_{\al(x)})$. This full subgroupoid, which we call the \emph{weak connected component} of $x$, has a natural involution induced from the involution $\al$. By choosing a representative for each set $\text{Ob}(A_x)\bigcup\text{Ob}(A_{\al(x)})$, one can display $A$ as a coproduct of its weakly connected components. 
\end{proof}

\begin{lem}
\label{lem:5.8}
Let $A$ be a groupoid together with an involution such that $A_{\f} = A$ and $w\colon B\acof C$ a projective acyclic cofibration in $\GGpd$. Then for any morphism $v\colon A\rightarrow C$ there exists a map $\widehat{v}\colon A\rightarrow B$ such that $w\circ \widehat{v} = v$.
\end{lem}
\begin{proof}
\label{proof:lem:5.8}	
We define $\widehat{v}$ as follows. Let $x$ be an element of $\text{Ob}(A)$. By the characterization of acyclic cofibrations \ref{prop:2.2}, there exists a unique $y\in B\f$ such that $v(x) = w(y)$. Take $\widehat{v}(x) = y$. Now, let $f\colon x\rightarrow x'$ be a morphism in $A$. Since $w$ is fully faithful, the induced map from $B(y,y')$ to $C(v(x),v(x'))$ is a bijection. Hence, there exists a unique map $\widehat{v}(f)$ such that $w(\widehat{v}(f)) = v(f)$. Note that $w(\be(\widehat{v}(f))) = v(\al(f))$, so by injectivity one has $\widehat{v}(\al(f)) = \be(\widehat{v}(f))$ as expected.
\end{proof}

\begin{lem}
\label{lem:5.9}
Let $f,g: A\rightarrow B$ be two right homotopic maps in $\GGpd$ such that $A_{\f} = A$. Then one has $f = g$.
\end{lem}
\begin{proof}
\label{proof:lem:5.9}
Indeed, let $PB$ be a path object for B
$$\xymatrix{B \ar@{ >->}[r]^{\sim}_{w} \ar[rd]_{\Delta} & PB \ar@{->>}[d] \\
	& B\times B\;,}$$
and $h$ a right homotopy between $f$ and $g$
$$\xymatrix{ & PB\ar@{->>}[d] \\
	A \ar[ru]^{h}\ar[r]_-{\langle f,g\rangle} & B\times B\;.}$$
By \ref{lem:5.8} applied with $C = PB$ and $v = h$, there exists $\widehat{h}$ such that $w\circ \widehat{h} = h$. So, we have $\Delta \circ \widehat{h} = \langle f,g\rangle$, hence $\widehat{h} = f = g$.
\end{proof}

\begin{lem}
	\label{righthomot}
	Two morphisms $f,g: A \rightarrow B$ in $\GGpd$ are right homotopic in the projective model structure if and only if there exists a $C_2$-equivariant natural isomorphism $\eta: f \iso g$ such that $\eta_a = 1_{f(a)}$ for every fixpoint $a$.
\begin{proof}\label{proof:righthomot}
	We first construct an explicit path object for any $A \in \GGpd$. Let $PA$ be the groupoid given by the following pushout
	$$\xymatrix{A\ar[r] \ar[d] & A^I \ar[d] \\ A^I \ar[r] & PA\pushoutcorner \:,}$$
	where $A^I$ is the standard path object of $\underline{A}$. The involution of $PA$ maps a triple of the form $(a,a,1_a)$ to $(\alpha(a), \alpha(a), 1_{\alpha(a)})$ and a triple of the form $(a,b,\varphi)$ with $\varphi \neq 1$ to $(\alpha(a), \alpha(b), \alpha(\varphi)$ in the other leg of the pushout. We have obvious equivariant morphisms $r: A \rightarrow PA$ and $(s,t): PA \rightarrow A \times A$, which are easily checked to be an acyclic cofibration and a fibration, respectively. \\
	Postcomposing $h$ with the obvious equivariant map $PB \rightarrow \text{Arr}(B)$ provides the required isomorphism $f \iso g$.  \\
	Conversely, assume there exists $\eta: f \iso g$ such that $\eta_a = 1_{f(a)}$ for every fixpoint $a$. We construct a lift $\hat{\eta}: A \rightarrow PB$ of $\langle f,g\rangle: A \rightarrow B\times B$ through $PB \rightarrow B\times B$ as follows. If $\eta_a = 1_{f(a)}$, then we take $\hat{\eta} = (f(a), f(a), 1_{f(a)})$. Otherwise, by assumption the set of objects remaining is a union of two-element $\alpha$-orbits. For each such orbit, choose arbitrarily one element $a$ and define $\hat{\eta}$ as $(f(a), g(a), \eta_a)$ and $\hat{\eta}(\alpha(a))$ as $(\beta(f(a)), \beta(g(a)), \beta(\eta_a))$ in the other leg of the pushout. The assignment of $\hat{\eta}$ on morphisms is straightforward.    
\end{proof}	
	
\end{lem}	

We give the following characterization of the right homotopy equivalences in $\GGpd$ with respect to the projective structure.
\begin{thm}
\label{thm:homotequiv}
	Let $f\colon A\rightarrow B$ be a morphism in $\GGpd$. The following are equivalent:
	\begin{enumerate}[label=(\roman*)]
		\item $f$ is a right homotopy equivalence.
		\item $\underline{f}$ is an equivalence of groupoids and induces an isomorphism between the full subgroupoids of fixed points $A\f$ and $B\f$.
		\item $\underline{f}$ is an equivalence of groupoids and induces an isomorphism between the subgroupoids of fixed points and fixed morphisms $A^{\Z2}$ and $B^{\Z2}$.
		\item $\underline{f}$ is an equivalence of groupoids and induces a bijection between the set of fixed points in $A$ and the set of fixed points in $B$.
	\end{enumerate}
\end{thm}
\begin{proof}
\label{proof:thm:homotequiv}
We will successively prove $(i)\Rightarrow (iv)$, $(iv)\Rightarrow (iii)$, $(iii)\Rightarrow (ii)$ and $(ii)\Rightarrow (i)$.
We prove $(i)\Rightarrow (iv)$. Assume $(i)$, so there exists $g\colon B\rightarrow A$ in $\GGpd$ such that $f\circ g\overset{r}{\sim}1_{B}$ and $g\circ f\overset{r}{\sim}1_{A}$. Hence, we have $(f\circ g)_{|B^{\Z2}}\overset{r}{\sim}1_{B^{\Z2}}$ and $(g\circ f)_{|A^{\Z2}}\overset{r}{\sim}1_{A^{\Z2}}$. So, by \ref{lem:5.9} one has $(f\circ g)_{|B^{\Z2}} = 1_{B^{\Z2}}$ and $(g\circ f)_{|A^{\Z2}} = 1_{A^{\Z2}}$. We conclude that $f^{\Z2}: A^{\Z2}\rightarrow B^{\Z2}$ is an isomorphism. Hence, in particular $f$ induces a bijection between the set of fixed points in $A$ and the set of fixed points in $B$. \\
We prove $(iv)\Rightarrow (iii)$. It is straighforward using the fact that $\underline{f}$ is in particular a fully faithful functor. \\
Next, we prove $(iii)\Rightarrow (ii)$. Note that $f_{|A\f}$ is bijective on objects, since $f^{\Z2}$ is so by assumption. Moreover, $f$ is fully faithful, hence $f_{|A\f}$ is an isomorphism. \\
Last, we prove $(ii)\Rightarrow (i)$. Note that by \ref{lem:weaklyconnected} we can assume without loss of generality that $A$ is weakly connected. Also, one can assume that $f$ is surjective. Indeed, first note that one can factorize $f$ through its image $\text{Im}f$, the full subgroupoid of $B$ whose objects are of the form $f(x)$ for some $x\in A$. The groupoid $\text{Im}f$ can be equipped with an involution thanks to $\be$. Indeed, given $y\in B$ such that there exists $x\in A$ and $f(x) = y$, then $f(\al(x)) = \be(f(x)) = \be(y)$. Second, we prove that the inclusion $\text{Im}f \hookrightarrow B$ is a projective acyclic cofibration. Indeed, since $(\text{Im}f)\f = B\f$ and $\text{Im}f$ is equivalent to $B$, we conclude by \ref{prop:2.3}. So, thanks to \ref{rmk:5.2} this inclusion is a right homotopy equivalence. One concludes by the \emph{2-out-of-3} property that $f\colon A\rightarrow B$ is a right homotopy equivalence if and only if $A \rightarrow\text{Im}f$ is so. The morphism $A \rightarrow\text{Im}f$ is still an equivalence of groupoids by the \emph{2-out-of-3} property and this morphism still induces an isomorphism between $A\f$ and $(\text{Im}f)\f$, since $(\text{Im}f)\f = B\f$. So, without loss of generality one can assume that our map $f$ is also surjective onto the objects (hence onto the morphisms). One wants to prove that $f$ is a homotopy equivalence. Below we will provide a morphism $g\colon B\rightarrow A$ in $\GGpd$ such that $f\circ g = 1_B$ and $g\circ f\overset{r}{\sim}1_A$.  \\
For each $b\in B$ we define an object $g(b)\in A$ as follows. If $b$ is a fixpoint, we take $g(b)$ to be the unique fixpoint $a$ with $f(a) = b$. For all the non-fixpoints, we choose one element $b$ out of each $\beta$-orbit, and define $g(b)$ to be any $a\in A$ satisfying $f(a) = b$ and define $g(\beta(b)) = \alpha(g(b))$. 
Now, we extend $g$ to be a functor. On maps, $g(\varphi: b \rightarrow b')$ is the unique morphism in $A(g(b), g(b'))$ such that $f (g(\varphi)) = \varphi$. This clearly makes $g$ $C_2$-equivariant. We know that $f\circ g = 1_B$, hence it remains to show $g\circ f\overset{r}{\sim}1_A$. So by \ref{righthomot} we just need to check that there is an equivariant natural isomorphism $\eta: g\circ f \iso 1_A$ such that $\eta_a = 1_a$ for every fixpoint $a$. If we take $\eta_a\in A(g(f(a)), a)$ to be the unique morphism satisfying $f(\eta_a) = 1_{f(a)}$, then the unicity and the equivariance of $f$ easily imply the equivariance of $\eta$ ({\em i.e.} $\alpha(\eta_a) = \eta_{\alpha(a)}$). Moreover, $\eta_a = 1_a$ for every fixpoint $a$.
\end{proof}

\section{The failure of univalence}
\label{sec:failure}

Tracing through the interpretation of type theory, one finds that the fibration $E \fib U\times U$, interpreting the dependent type of weak equivalences, is such that the set of objects of the fiber over a pair $(x,y)\in U\times U$ is the set of isomorphisms in $U$ between $x$ and $y$. Moreover, the involution on $E$ maps $(x,y,\varphi)$, where $\varphi$ is an isomorphism from $x$ to $y$, to $(\upsilon(x),\upsilon(y),\upsilon(\varphi))$ (with $\upsilon$ the involution on $U$).

\begin{prop}
\label{prop:failure}
	Univalence does not hold for the universe $p\colon \widetilde{U}\rightarrow U$ (\textit{cf.} \ref{sec:universes}) in the type-theoretic fibration category given in \ref{cor:projectivettfc}.
\end{prop}
\begin{proof}
\label{proof:prop:failure}
	The morphism $U\to E$ (\textit{cf.} the end of section \ref{univalenceproperty}) is defined as follows 
	$$\ffour{U}{E}{x}{(x,x,1_x)\;,}$$
	\textit{i.e.} it maps an object $x$ to the identity isomorphism of $x$ in $U$. Note that this morphism is not surjective onto the fixed points of $E$. Indeed, it is easy to find a non-trivial fixed point of $E$.
	For instance, take the following elements of $U\colon (\mathbb{N},\mathbb{N},1_{\mathbb{N}})$, $(2\mathbb{N},2\mathbb{N},1_{2\mathbb{N}})$,
	and the isomorphism $(2n,2n)$ between them, where $2n$ denotes the bijection from $\mathbb{N}$ to $2\mathbb{N}$ that maps $n$ to $2n$. Then this triplet is a fixed point of $E$, where the third component is not the identity. So, it does not belong to the image of the morphism above. Note that we can even take two identical small groupoids and still find a fixed point of $E$ that does not belong to the image of $U\rightarrow E$. Indeed, consider $(\mathbb{Z},\mathbb{Z},1_{\mathbb{Z}})$ and the automorphism $(-n,-n)$, where $-n$ denotes the bijection from $\mathbb{Z}$ to $\mathbb{Z}$ that maps $n$ to $-n$. So, according to \ref{thm:homotequiv} the map $U\rightarrow E$ is not a right homotopy equivalence, hence univalence does not hold in the projective type-theoretic fibration structure on $\GGpd$.
\end{proof}

Below we investigate whether function extensionality holds. For details about the meaning of function extensionality in a type-theoretic fibration category see \cite[section 5]{shulman:invdia}. In particular, according to \cite[lemma 5.9]{shulman:invdia} function extensionality holds in the internal language of a type-theoretic fibration category if and only if dependent products along fibrations preserve acyclic fibrations (\textit{i.e.} fibrations that are also right homotopy equivalences).

\begin{notn}
	\label{notn:5.6}	
	In the rest of this section we will use the following notations:
	\begin{align*}
	S(\1) & \coloneqq \1\textstyle\coprod\1 \coloneqq \xymatrix{0 & 0'} \\
	S(\bI) & \coloneqq \bI\textstyle\coprod\bI \coloneqq  \xymatrixrowsep{0.5pc}\xymatrix{0 \ar[r]^{\phi} & 1 \;,\\
		0' \ar[r]^{\phi'} & 1'}
	\end{align*}
	both equipped with the swap involution.
\end{notn}	

\begin{prop}
\label{prop:funext}
	Function extensionality does not hold in the internal type theory of the projective type-theoretic fibration structure on $\GGpd$.
\end{prop}
\begin{proof}
\label{proof:prop:funext}
	It suffices to prove that there exist a fibration $g$ and a fibration $f$ such that $f$ is a right homotopy equivalence and $\Pi_{g}f$ is not a right homotopy equivalence. In order to achieve this, consider $g\colon S(\1)\rightarrow\mathbf{1}$ and $f\colon S(\bI)\rightarrow S(\1)$ (see \ref{notn:5.6}) in $\GGpd$. The reader can easily check that $\underline{f}$ is fully faithful and surjective (and so it is an acyclic fibration in $\Gpd$), and $f$ restricted to fixed points is the identity (since $(S(\bI))\f= (S(\1))\f = \varnothing$). So, according to \ref{thm:homotequiv} $f$ is a right homotopy equivalence. Now, since $\Pi_{g}f$ goes from $\text{dom}(\Pi_{g}f)$ to $\mathbf{1}$, it suffices to prove that $\text{dom}(\Pi_{g}f)$ has at least two fixed points. Using the explicit construction of the right adjoint $\Pi_g$ given in \cite[lemma 4.3.4]{bordg:thesis}, one notices that a fixed point of $\text{dom}(\Pi_{g}f)$ over $\mathbf{1}$ is nothing but a section $s$ of $f$ such that $\pi_{g}f(s) = s$ (where $\pi_gf$ is the involution on $\Pi_gf$). But we have two such sections $s_1$ and $s_2$. Indeed, take 
	$$\fsix{s_1\colon S(\1)}{S(\bI)}{0}{0}{0'}{0'}$$  and 
	$$\fsix{s_2\colon S(\1)}{S(\bI)}{0}{1}{0'}{1'\;.}$$
\end{proof}

\begin{rmk}
	We recall that the univalence axiom implies function extensionality for types in the universe. 
	However, since the above proposition involves non-discrete groupoids, it does not give us an alternative proof that univalence does not hold for the universe $p\colon \widetilde{U}\rightarrow U$. Also, note that according to \cite[Remark 5.10]{shulman:invdia} function extensionality holds in the natural type-theoretic fibration structure on $\Gpd$. So, with the projective type-theoretic fibration structure on $\GGpd$ function extensionality is also broken.
\end{rmk}

\section{Conclusion}
\label{sec:conclusion}

Our work suggests that projective fibrations are a good choice in order to lift a bare type-theoretic fibration structure from a category $\mathscr{C}$ to a functor category $[\mathcal{D},\mathscr{C}]$ even in the presence of non-trivial isomorphisms in $\mathcal{D}$, and eventually to provide the additional structure needed for universes. But, projective fibrations do not seem to be adequate for the stability of the univalence property due to the lack of projective homotopy equivalences. \\
Nevertheless, the model presented in this article provides a new model of type theory with dependent sums, dependent products, identity types and a universe. To the best of our knowledge it was the first model derived from a Quillen model structure where not all objects are cofibrant. If it should happen in a model that not all objects are fibrant-cofibrant, then our method of proof makes it clear that even when a whole model structure is available at hand only the classes of fibrations, acyclic cofibrations and right homotopy equivalences are relevant in the context of type theory. \\
Moreover, this model together with the model in \cite{Bordg2019}, using the Quillen equivalent injective model structure on the same bare category, gives a counterexample to a tentative \emph{model invariance problem}\footnote{\url{https://ncatlab.org/homotopytypetheory/show/model+invariance+problem}} suggested by Michael Shulman. \\
Last, as explained in \cite{lmcs:3217} the definition of propositional truncation in Martin-L\"of type theory is an open problem$\colon$ ``Is weak propositional truncation definable in Martin-L\"of type theory ? This is commonly believed to not be the case. However, the standard models do have propositional truncation, making it hard to find a concrete proof.'' (p.34 \textit{ibid.}). One could check whether the propositional truncation holds in this weird projective model in $\GGpd$. In the negative, it would provide a way to solve this open problem.

\bibliographystyle{alpha}
\bibliography{all}

\end{document}